\documentclass[a4paper,12pt]{article}
\input{mymacros.sty}

\newcommand{\Gmax}{G_\mathrm{max}}
\newcommand{\rhonat}{\rho_{\mathrm{Nat}}}
\newcommand{\can}{\text{can}}
\newcommand{\AHilb}{\mathop{A\text{-}\mathrm{Hilb}}\nolimits}
\newcommand{\GHilb}{\mathop{G\text{-}\mathrm{Hilb}}\nolimits}
\newcommand{\GoHilb}{\mathop{G_0\text{-}\mathrm{Hilb}}\nolimits}
\newcommand{\Nat}{{\mathrm{Nat}}}
\newcommand{\rep}{\operatorname{rep}}

\title{The special McKay correspondence\\
and exceptional collections}
\author{Akira Ishii and Kazushi Ueda}
\date{}
\begin{document}

\maketitle

\begin{abstract}
We show that the derived category of coherent sheaves
on the quotient stack of the affine plane
by a finite small subgroup of the general linear group
is obtained from the derived category of coherent sheaves
on the minimal resolution by adding a semiorthogonal summand
with a full exceptional collection.
The proof is based on an explicit construction
in the abelian case,
together with the analysis of the behavior
of the derived categories of coherent sheaves
under root constructions.
\end{abstract}

\section{Introduction}

Let $G$ be a finite small subgroup of $\GL_2(\bC)$
acting  on the affine plane $\bA^2 = \Spec R$.
%with the coordinate ring $R = \bC[x,y]$.
The quotient singularity $X = \bA^2/G = \Spec R^G$
has two kinds of natural resolutions:
One is the minimal resolution $\tau : Y \to X$,
which exists uniquely
by the minimal model theory in dimension two.
The other is the non-commutative ring
$A = \End_{R^G} R$,
which is a {\em non-commutative crepant resolution}
in the sense of Van den Bergh
\cite[Definition 4.1]{Van_den_Bergh_NCR}.

The minimal resolution $Y$ is crepant
%if and only if $X$ is a Kleinian singularity, i.e.,
if and only if $G$ is a subgroup of $\SL_2(\bC)$,
whereas the non-commutative resolution $A$ is always crepant.
The ring $A$ is Morita equivalent to the crossed-product algebra
$R \rtimes G$,
so that the category of finitely-generated $A$-modules
is equivalent to the category of $G$-equivariant coherent sheaves
on $\bA^2$,
which in turn is equivalent to the category of coherent sheaves
on the quotient stack $[\bA^2 / G]$;
%by definition;
\begin{align*}
 \module A
  \cong \module R \rtimes G
  \cong \coh [\bA^2/G].
\end{align*}

When $G$ is a subgroup of $\SL_2(\bC)$,
Ito and Nakamura \cite{Ito-Nakamura_HSSS}
%gave a fine moduli interpretation of
constructed
the commutative crepant resolution $Y$
as the $G$-Hilbert scheme
\cite{Nakamura_HSAGO}
%Bridgeland-King-Reid}
parametrizing
%A {\em $G$-cluster} is 
$G$-invariant subschemes
$Z \subset \bA^2$
such that $H^0(\scO_Z)$ is isomorphic
to the regular representation of $G$
as a $G$-module.
This fine moduli interpretation
comes with the universal flat family
\begin{align} \label{eq:univ_family}
\begin{CD}
 \scZ @>{q}>> \bA^2 \\
 @V{p}VV @VV{\pi}V \\
 Y @>{\tau}>> X,
\end{CD}
\end{align}
which allows one to define
the integral functor
\begin{align} \label{eq:Phi}
 \Phi = q_* \circ p^* :
   D^b \coh Y \to D^b \coh [\bA^2 / G]
\end{align}
realizing the McKay correspondence
as an equivalence of derived categories
\cite{Kapranov-Vasserot, Bridgeland-King-Reid}.
This provides an example of
a generalization \cite[Conjecture 4.6]{Van_den_Bergh_NCR}
of a conjecture of Bondal and Orlov
\cite{Bondal-Orlov_semiorthogonal}
that any crepant resolutions of $X$,
either commutative or non-commutative,
are derived equivalent.

Even if $G$ is not a subgroup of $\SL_2(\bC)$,
the Hilbert-Chow morphism $\tau$
in the diagram \eqref{eq:univ_family}
is still a resolution of $X$,
which is minimal but not crepant
\cite{Ishii_MKG}.
The integral functor $\Phi$ is not an equivalence
but a full and faithful embedding,
and its essential image is admissible
\cite[Definition 2.1]{Bondal-Orlov_semiorthogonal}
since $\Phi$ has both left and right adjoints.

The essential image of $\Phi$
and its right orthogonal
are described as follows:

\begin{proposition} \label{pr:non-special}
Let $G$ be a finite subgroup of $\GL_2(\bC)$ and
$Y$ be the Hilbert scheme of $G$-orbits in $\bA^2$.
Then the essential image of $\Phi$ is generated by
$\{ \scO_{\bA^2} \otimes \rho \}_{\rho : \text{special}}$,
and its right orthogonal is generated by
$\{ \scO_0 \otimes \rho \}_{\rho : \text{non-special}}$.
\end{proposition}

{\em Special representations} are introduced
by Wunram \cite{Wunram}
to extend the McKay correspondence
to subgroups of $\GL_2(\bC)$.
We recall the basic definitions and properties
of special representations
in Section \ref{sc:special-McKay},
where the proof of Proposition \ref{pr:non-special}
is also given.

In the case of cyclic groups, we can prove
the existence of a full exceptional collection
in the semiorthogonal complement
of the essential image of $\Phi$:

\begin{theorem} \label{th:cyclic1}
Let $G$ be a finite small cyclic subgroup of $\GL_2(\bC)$ and
$Y$ be the Hilbert scheme of $G$-orbits in $\bA^2$.
Then there is an exceptional collection
$
 (E_1, \dots, E_n)
$
in $D^b \coh [\bA^2 / G]$ and a semiorthogonal decomposition
$$
 D^b \coh [\bA^2 / G]
  = \langle E_1, \dots, E_n, \Phi(D^b \coh Y) \rangle,
$$
where $n$ is the number of
irreducible non-special representations of $G$.
\end{theorem}

Theorem \ref{th:cyclic1} is not obvious at all, since
\begin{itemize}
 \item
the set $\{\scO_0 \otimes \rho \}_{\rho : \text{non-special}}$
rarely form an exceptional collection
%even in the case of a cyclic group
(cf. Example \ref{eg:83}), and
 \item
the category $D^b \coh [\bA^2/G]$ does not have
an exceptional object at all
when $G$ is a subgroup of $\SL_2(\bC)$.
\end{itemize}

%An example where
%$\{\scO_0 \otimes \rho \}_{\rho : \text{non-special}}$
%do not form an exceptional collection
%is given in Example \ref{eg:83}.
We use the abelian case to obtain a similar result
in a general case by using a slightly different functor,
while we expect the same result for the functor $\Phi$.

\begin{theorem}\label{th:main}
Let $G$ be a finite small subgroup of $\GL_2(\bC)$ and
$Y \to \bA^2/G$ be the minimal resolution of $\bA^2/G$.
For a suitable fully faithful functor
$$
 \Phi' : D^b \coh Y \to D^b \coh [\bA^2/G],
$$
there is an exceptional collection
$
 (E_1, \dots, E_n)
$
in $D^b \coh [\bA^2 / G]$ and a semiorthogonal decomposition
$$
 D^b \coh [\bA^2 / G]
  = \langle E_1, \dots, E_n, \Phi'(D^b \coh Y) \rangle,
$$
where $n$ is the number of
irreducible non-special representations of $G$.
\end{theorem}

Theorem \ref{th:main} is complementary to the works of
Craw \cite{Craw_SMC} and Wemyss \cite{Wemyss_GL2},
which describe $D^b \coh Y$
as the derived category of modules
over the path algebra of a quiver with relations
called the {\em special McKay quiver}.
%Their works in turn are based on a result of
%Van den Bergh \cite[Theorem B]{Van_den_Bergh_TFNR}.
One can say that
their works give a non-commutative description
of the commutative non-crepant resolution,
whereas Theorem \ref{th:main} gives the relation
between the commutative non-crepant resolution
and the non-commutative crepant resolution.

%%%
%%%

We now give the definition of the functor $\Phi'$.
The action of $G$ on $\bA^2$ induces
\begin{itemize}
 \item
an action of $G_0:=G \cap \SL(2, \bC)$ on $\bA^2$, and
 \item
an action of $G/G_0$ on $G_0\text{-Hilb}(\bA^2)$.
\end{itemize}
The Hilbert-Chow morphism
$Y_2 := G/G_0\text{-Hilb}(G_0\text{-Hilb}(\bA^2)) \to \bA^2/G$
from the iterated Hilbert scheme
is a resolution of $\bA^2/G$.
The resolution $Y_2 \to \bA^2/G$ is not necessarily minimal,
and factors through the minimal resolution
$Y \to \bA^2/G$;
\vspace{3mm}
$$
\begin{psmatrix}
 Y_2 & & Y \\
  & \bA^2/G.
\end{psmatrix}
\psset{arrows=->,shortput=nab,nodesep=5pt}
\ncline{1,1}{1,3}^{\varphi}
\ncline{1,1}{2,2}
\ncline{1,3}{2,2}
$$
By embedding $G$ into $\SL_3(\bC)$ and
embedding $Y_2$ as a divisor in $G/G_0\text{-Hilb}(G_0\text{-Hilb}(\bA^3))$,
one can deduce from \cite[Theorem 2.7]{MR3049308}
that $Y_2$ can be identified with the moduli space $\scM_\theta$
of stable $G$-equivariant sheaves on $\bA^2$
for a suitable choice of a stability parameter $\theta$.
This gives a fully faithful functor
$$
 \Phi_2'(-) := {\pi_2}_* \lb {\pi_1}^*(-) \otimes \scE_\theta \rb :
  D^b \coh Y_2 \to D^b \coh [\bA^2 / G],
$$
where $\scE_\theta$ is the universal family on $\scM_\theta \times [\bA^2/G]$.
The composition
$$
 \Phi' := \Phi_2' \circ \varphi^* : D^b \coh Y \to D^b \coh [\bA^2/G]
$$
is a fully faithful functor.

%%%
%%%
The proof of Theorem \ref{th:main} proceeds
as follows:
%consists of several steps:
\begin{enumerate}
 \item \label{it:cyclic}
If $G \subset \GL_2(\bC)$ is a cyclic group,
then
%$\GHilb \bA^2$ is a toric variety, and
special representations can be computed by continued fraction expansions
\cite{Wunram2, Wunram},
and we can explicitly construct an exceptional collection $E_1, \dots, E_n$
in $\coh[\bA^2/G]$ as in Theorem \ref{th:cyclic}.
\footnote{
Kawamata pointed out that
this step can also be carried out
using his arguments
\cite{Kawamata_LCBMDC, Kawamata_DCTV},
and subsequently written a paper \cite{Kawamata_DCTVII}
which includes it as a special case.}
 \item \label{it:equivariant}
Let $G$ be a finite small subgroup of $\GL_2(\bC)$
and put $G_0 = G \cap \SL_2(\bC)$.
Then $G_0$ is a normal subgroup of $G$
and $A = G / G_0$ is a cyclic group.
The group $A$ acts on $Y_0 = \GoHilb \bA^2$ and
one has an equivalence
\begin{align} \label{eq:SOD0}
 \Phi_0 : D^b \coh [Y_0 / A] \simto D^b \coh [\bA^2 / G] 
\end{align}
by Theorem \ref{th:equivariant},
which is an equivariant version of the McKay correspondence
\cite{Kapranov-Vasserot, Bridgeland-King-Reid}.
Since $Y_0$ is a resolution of $\bA^2 / G_0$,
a resolution of $Y_0 / A$ is a resolution of $\bA^2 / G$.
\item \label{it:root}
The stack $[Y_0/A]$ may have non-trivial
stabilizer groups along divisors,
whereas the {\it canonical stack} $\scY_1$
associated with the coarse moduli space $Y_1 := Y_0/A$
is a stack which has trivial stabilizer groups
except at the singular points.
There is a morphism $[Y_0/A] \to \scY_1$
coming from the universal property
of the canonical stack,
which can be regarded as an iteration of {\it root constructions}
\cite{Abramovich-Graber-Vistoli, Cadman_US}
along simple normal crossing divisors.
The coarse moduli spaces
of irreducible divisors with non-trivial stabilizer groups
are smooth rational curves,
so that one has a full and faithful functor
$
 \Phi_1 : D^b \coh \scY_1 \to D^b \coh [Y_0/A]
$
and a semiorthogonal decomposition
%$$
% D^b \coh [\bA^2 / G]
%  = \langle E_1, \dots, E_{N_1}, D^b \coh Y_1 \rangle
%$$
\begin{align} \label{eq:SOD1}
 D^b \coh [Y_0 / A]
  = \langle E_1, \dots, E_{n_1}, \Phi_1(D^b \coh \scY_1) \rangle
\end{align}
by Proposition \ref{pr:root}.
% \item
%Some irreducible components of
%the exceptional divisor of $Y_0$
%may be fixed pointwise
%by the action of a subgroup $B$ of $A$.
%If this is the case, then the quotient stack $[Y_0 / A]$ contains
%an open substack which is isomorphic to $[\scO_{\bP^1}(-2) / B]$.
%Let $Y_1$ be the stack obtained by replacing it
%with its coarse moduli space $\scO_{\bP^1}(- 2 |B|)$.
%This gives a semiorthogonal decomposition
%$$
% D^b \coh [\bA^2 / G]
%  = \langle E_1, \dots, E_{N_1}, D^b \coh Y_1 \rangle.
%$$
%Note that $Y_1$ may no longer be a global quotient stack,
%although it is still a smooth Deligne-Mumford stack.
%If $Y_1$ still has a divisor with a generic stabilizer,
%then one can repeat the same procedure,
%so that one can assume $Y_1$ has no divisor
%with a generic stabilizer.
 \item \label{it:cyclic2}
The coarse moduli space $Y_1$ of $\scY_1$ has
cyclic quotient singularities.
By taking the minimal resolution of it,
we obtain a resolution $Y_2$ of $\bA^2 / G$.
This gives a full and faithful functor
$
 \Phi_2 : D^b \coh \scY_1 \to D^b \coh Y_2
$
and a semiorthogonal decomposition
%$$
% D^b \coh [\bA^2 / G]
%  = \langle E_1, \dots, E_{N_2}, D^b \coh Y_2 \rangle
%$$
\begin{align} \label{eq:SOD2}
 D^b \coh \scY_1
  = \langle E_{n_1+1}, \dots, E_{n_2},
   \Phi_2(D^b \coh Y_2) \rangle
\end{align}
by Proposition \ref{pr:orbifold}.
 \item \label{it:minimal}
The minimal resolution $Y$ can be obtained from $Y_2$
by contracting $(-1)$-curves.
This gives a full and faithful functor
$
 \Phi_3 : D^b \coh Y \to D^b \coh Y_2
$
and a semiorthogonal decomposition
%$$
% D^b \coh [\bA^2 / G]
%  = \langle E_1, \dots, E_{N}, D^b \coh Y \rangle
%$$
\begin{align} \label{eq:SOD3}
 D^b \coh Y_2
  = \langle E_{n_2+1}, \dots, E_{n},
   \Phi_3(D^b \coh Y) \rangle
\end{align}
by Orlov \cite[Theorem 4.3]{Orlov_PB}.
\end{enumerate}

By combining 
the semiorthogonal decompositions
from \eqref{eq:SOD0} to \eqref{eq:SOD3},
one obtains Thoerem \ref{th:main}.
In fact, our proof of Theorem \ref{th:main}
readily gives the following global analog:

\begin{theorem} \label{th:global}
Let $\scX$ be the canonical stack
associated with a surface $X$
with at worst quotient singularities,
and $Y$ be the minimal resolution of $X$.
Then there is a full and faithful functor
$$
 \Phi : D^b \coh Y \to D^b \coh \scX
$$
and a semiorthogonal decomposition
\begin{equation*}
 D^b \coh \scX =
  \langle
   E_1, \dots, E_{\ell},  \Phi (D^b \coh Y)
  \rangle
\end{equation*}
where $E_1, \dots, E_{\ell}$ is an exceptional collection.
\end{theorem}

This gives the relation
between the derived categories of the commutative minimal resolution
and a non-commutative crepant resolution
for any surface $X$ with at worst quotient singularities.
As an application of Theorem \ref{th:global},
we show the existence
of a full exceptional collection
on a two-dimensional stack
associated with an invertible polynomial
in Theorem \ref{th:invertible}.

Root constructions
appearing in Step \ref{it:root} %above
are introduced independently
by Cadman \cite{Cadman_US} and
Abramovich, Graber and Vistoli \cite{Abramovich-Graber-Vistoli},
and play important roles
in the theory of toric stacks
\cite{Borisov-Chen-Smith, Fantechi-Mann-Nironi}
and orbifold Gromov-Witten theory
\cite{Abramovich-Graber-Vistoli}.
As the analysis of the derived categories of root stacks
in Step \ref{it:root} may also be of independent interest,
we state it as theorems here.
The first result concerns the root stack of a line bundle:

\begin{theorem} \label{th:root1}
Let $\scL$ be a line bundle
on a Deligne-Mumford stack $\scX$
and $\sqrt[r]{\scL/\scX}$ be the $r$-th root stack
for a positive integer $r$.
Then the abelian category of coherent sheaves
on $\sqrt[r]{\scL/\scX}$
is the direct sum of $r$ copies
of the abelian category of coherent sheaves on $\scX$;
\begin{align*}
 \coh \sqrt[r]{\scL/\scX} \cong \lb \coh \scX \rb^{\oplus r}.
\end{align*}
\end{theorem}

Note that the decomposition above is
not only semiorthogonal
but orthogonal,
and we do not need to pass to the derived categories.
Theorem \ref{th:root1} enables us
to generalize the results of Borisov and Hua
\cite{Borisov-Hua}
to the case when the $N$-lattice has torsion
(cf. the second paragraph in \cite[Section 2]{Borisov-Hua}).

The second result deals with the root stack of a line bunlde with a section:

\begin{theorem} \label{th:root2}
Let $\scD$ be a smooth divisor
in a smooth Deligne-Mumford stack $\scX$
and $\scY = \sqrt[r]{(\scO_\scX(\scD), 1) / \scX}$
be the $r$-th root stack of the line bundle $\scO_\scX(\scD)$
with the canonical section $1 \in H^0(\scO_\scX(\scD))$.
Then there are full and faithful functors
\begin{align*}
\Phi_\scX &: D^b \coh \scX \to D^b \coh \scY, \\
\Phi_\scD &: D^b \coh \scD \to D^b \coh \scY
\end{align*}
and a semiorthogonal decomposition
\begin{align*}
 D^b \coh \scY = \la
  \Phi_\scD(D^b \coh \scD) \otimes \scM^{\otimes (r-1)},
   \ldots
  \Phi_\scD(D^b \coh \scD) \otimes \scM,
  \Phi_\scX(D^b \coh \scX) \ra,
\end{align*}
where $\scM$ is the universal line bundle on $\scY$.
\end{theorem}

We assume that all divisors are Cartier
throughout this paper.
Theorem \ref{th:root2} is a root stack analog of
\cite[Theorem 4.3]{Orlov_PB},
where the derived category of the blow-up
is described in terms of derived categories
of the original variety and the center.
This shows that the root construction
behaves very much like the `blow-up along a divisor'
as long as derived categories of coherent sheaves are concerned.

This paper is organized as follows:
We recall the definition of special representations
and prove Proposition \ref{pr:non-special}
in \pref{sc:special-McKay}.
Steps \ref{it:cyclic} and \ref{it:equivariant} are carried out
in Sections \ref{sc:cyclic} and \ref{sc:equivariant} respectively.
\pref{th:root1} is proved in \pref{sc:root1}, and
\pref{th:root2} is proved in \pref{sc:root2}.
Steps \ref{it:root} and \ref{it:cyclic2} are carried out
in Sections \ref{sc:iteration} and \ref{sc:canonical} respectively.
%Step \ref{it:minimal}
Theorems \ref{th:main} and \ref{th:global}
are proved in \pref{sc:proof}.
As a corollary,
we show in Section \ref{sc:invertible}
that the two-dimensional Deligne-Mumford stack
associated with an invertible polynomial in four variables
has a full exceptional collection.

{\bf Acknowledgment}:
We thank Yujiro Kawamata
for the remark on Step 1 above.
A.~I. is supported by Grant-in-Aid for Scientific Research (No.18540034).
K.~U. is supported by Grant-in-Aid for Young Scientists
(No.20740037 and No.24740043).
\section{The special McKay correspondence}
 \label{sc:special-McKay}

In this section,
we recall the definition of special representations
and prove Proposition \ref{pr:non-special}.
Let $G$ be a finite small subgroup of $\GL_2(\bC)$
acting  on the affine plane $\bA^2 = \Spec R$
and $\pi : Y \to X = \Spec R^G$ be the minimal resolution
of the quotient singularity.
First we recall the relation
between full sheaves on $Y$
and reflexive modules on $X$:

\begin{definition-lemma}[Esnault \cite{Esnault_RMQSS}]
Let $\scM$ be a sheaf on $Y$
and $\scM^{\vee}$ be its dual sheaf.
Then there exists a reflexive module $M$ on $X$ such that
$
 \scM \cong \Mtilde := \pi^*M / \text{\it torsion}
$
if and only if the following three conditions are satisfied:
\begin{enumerate}
 \item
$\scM$ is locally-free.
 \item
$\scM$ is generated by global sections.
\item
$H^1((\scM)^{\vee} \otimes \omega_{Y})=0$.
\end{enumerate}
In this case $\scM$ is said to be {\em full}.
\end{definition-lemma}

%On the other hand,
%reflexive sheaves on $X$
%correspond to representations of $G$.
Note that reflexive modules coincide
with Cohen-Macaulay modules
since $X$ is a normal surface.

\begin{theorem}[{Auslander \cite{Auslander_RSASS}}]
The functor $(-)^G$ of taking $G$-invariant part
gives an equivalence
from the category of projective $R \rtimes G$-modules
to the category of Cohen-Macaulay $R^G$-modules.
\end{theorem}

It follows that
indecomposable full sheaves on $Y$ are
in one-to-one correspondence
with irreducible representations of $G$.

\begin{theorem}[{Wunram \cite[Main Result]{Wunram}}]
 \label{th:Wunram-1}
Let
$
 E = \bigcup_{i=1}^r E_i
$
be the decomposition into irreducible components
of the exceptional set $E$.
Then for every curve $E_i$
there exists exactly one indecomposable reflexive module
$M_i$
such that the corresponding full sheaf
$\Mtilde_i = \pi^* M_i / \text{\it torsion}$
satisfies the conditions $H^1((\Mtilde)^{\vee})=0$
and
$$
 c_1(\Mtilde_i) \cdot E_j = \delta_{ij}.
$$
\end{theorem}

A full sheaf is said to be {\em special}
if there is an index $1 \le i \le r$ such that $\scM = \scM_i$
or it is isomorphic to the structure sheaf $\scO_Y$.
The special full sheaf $\scO_Y$ corresponds to the trivial representation
and is denoted by $\scM_0$.
Special full sheaves are characterized as follows:

\begin{theorem}[{Wunram \cite[Theorem 1.2]{Wunram}}]
 \label{th:Wunram-2}
An indecomposable full sheaf $\scM$ is special
if and only if $H^1(\scM^{\vee})=0$.
\end{theorem}

An irreducible representation $\rho$ of $G$ is said to be special
if the corresponding full sheaf
$
 \scM_\rho
  = \pi^*
     \left(
      (\rho^\vee \otimes R)^G
     \right) / \text{torsion}
$
is special.
%Note that $\scM_\omega$ is isomorphic to the dualizing sheaf of $Y$,
%where $\omega = \det(\rhonat^\vee)$ is the determinant
%of the dual of the natural representation
%$
% \rhonat : A \hookrightarrow \GL_2(\bC).
%$
%
%\begin{theorem}[{Wunram \cite[Theorem 1.2]{Wunram}}] \label{th:Wunram-3}
%An irreducible representation $\rho$ of $A$ is special
%if and only if the natural inclusion map
%$
% \scM_{\rho} \otimes \scM_{\omega}
%  \to \scM_{\rho\otimes\omega}
%$
%is an isomorphism.
%\end{theorem}

Special full sheaves generate the derived category
of coherent sheaves on $Y$:

\begin{theorem}[{Van den Bergh \cite[Theorem B]{Van_den_Bergh_TFNR}}]
 \label{th:VdB}
The direct sum
%$
% \displaystyle{\bigoplus_{\rho : \text{special}}}
%   \scM_\rho
%$
of indecomposable special full sheaves
generates $D^b \coh Y$.
\end{theorem}

\begin{proof}[Proof of Proposition \ref{pr:non-special}]
It follows from Theorem \ref{th:VdB} that
the essential image of $\Phi$ is generated by
$
 \{ \scO_{\bA^2} \otimes \rho \}_{\rho : \text{special}}.
$

$
 \{ \scO_0 \otimes \rho \}_{\rho : \text{non-special}}
$
is right orthogonal to
$
 \{ \scO_{\bA^2} \otimes \rho \}_{\rho : \text{special}}
$
since
$$
 \RHom_{[\bA^2/G]}(\scO_{\bA^2} \otimes \rho,
  \scO_0 \otimes \tau)
  \cong
\begin{cases}
 \bC & \rho = \tau, \\
 0 & \text{otherwise}. 
\end{cases}
$$
Together, they generate $D^b \coh [\bA^2/G]$.
\end{proof}

Let $\rhonat$ be the two-dimensional representation of $G$
coming from the inclusion $G \subset \GL_2(\bC)$,
and $a_{\mu \nu}$ be the multiplicity
appearing in the irreducible decomposition
$$
 \mu \otimes \rhonat
  = \bigoplus_{\nu \in \Irrep(G)} \nu^{\oplus a_{\mu \nu}}
$$
of tensor products of in the representation ring of $G$.
It follows from the projective resolution
$$
\begin{CD}
 0 @>>>
  \scO_{\bA^2} \otimes \det \rhonat
   @>>>
  \scO_{\bA^2} \otimes \rhonat
   @>>>
  \scO_{\bA^2}
   @>>>
  \scO_0
   @>>> 0
\end{CD}
$$
that one has
\begin{align*}
 \dim \Hom(\scO_0 \otimes \mu, \scO_0 \otimes \nu)
  = \delta_{\mu \nu}, \\
 \dim \Ext^1(\scO_0 \otimes \mu, \scO_0 \otimes \nu)
  = a_{\mu \nu},
\end{align*}
and
$$
 \dim \Ext^2(\scO_0 \otimes \mu, \scO_0 \otimes \nu)
  = \dim \Hom(\scO_0 \otimes \nu,
   \scO_0 \otimes \mu \otimes \det \rho_\Nat).
$$
This is summarized in the {\em McKay quiver} of $G$,
whose vertices are irreducible representations of $G$
whose solid arrows from $\mu$ to $\nu$ are basis of
$\Ext^1(\scO \otimes \mu, \scO \otimes \nu)$,
and whose dashed arrows are basis of
$\Ext^2(\scO \otimes \mu, \scO \otimes \nu)$.

\begin{example} \label{eg:83}
As an example,
consider the case
when $G = \la \frac{1}{8}(1, 3) \ra$,
whose McKay quiver is shown
in Figure \ref{fg:McKay_quiver},
and its full subquiver consisting of non-special vertices
is shown in Figure \ref{fg:non-special_quiver}.
This clearly shows that the set
$\{ \scO_0 \otimes \rho \}_{\rho : \text{non-special}}$
does not form an exceptional collection.
\begin{figure}[htbp]
\begin{minipage}{.5 \linewidth}
\centering
\begin{psmatrix}[mnode=circle,colsep=1cm,rowsep=1cm]
 & 1 &[colsep=1.7cm] 0 \\
 2 & & & 6 \\[.5cm]
 3 & & & 7 \\
 & 4 &[colsep=1.7cm] 5
\end{psmatrix}
\psset{arrows=->,nodesep=1pt}
\ncline{1,2}{2,1}
\ncline{2,1}{3,1}
\ncline{3,1}{4,2}
\ncline{4,2}{4,3}
\ncline{4,3}{3,4}
\ncline{3,4}{2,4}
\ncline{2,4}{1,3}
\ncline{1,3}{1,2}
\ncline{1,2}{4,2}
\ncline{2,1}{4,3}
\ncline{3,1}{3,4}
\ncline{4,2}{2,4}
\ncline{4,3}{1,3}
\ncline{3,4}{1,2}
\ncline{2,4}{2,1}
\ncline{1,3}{3,1}
\psset{arrows=->,linestyle=dashed,offset=2pt,nodesep=1pt}
\ncline{1,3}{4,2}
\ncline{1,2}{4,3}
\ncline{2,1}{3,4}
\ncline{3,1}{2,4}
\ncline{4,2}{1,3}
\ncline{4,3}{1,2}
\ncline{3,4}{2,1}
\ncline{2,4}{3,1}
\caption{The McKay quiver}
\label{fg:McKay_quiver}
\end{minipage}
\begin{minipage}{.5 \linewidth}
\centering
\begin{psmatrix}[colsep=1cm,rowsep=1cm,mnode=circle]
 &[linecolor=lightgray] \textcolor{lightgray}{1}
 &[colsep=1.7cm,linecolor=lightgray] \textcolor{lightgray}{0} \\
 2 & & & 6 \\[.5cm]
 [linecolor=lightgray] \textcolor{lightgray}{3}
 & & & 7 \\
 & 4 &[colsep=1.7cm] 5
\end{psmatrix}
\psset{arrows=->,nodesep=1pt}
\ncline{4,2}{4,3}
\ncline{4,3}{3,4}
\ncline{3,4}{2,4}
\ncline{2,1}{4,3}
\ncline{4,2}{2,4}
\ncline{2,4}{2,1}
\psset{arrows=->,linestyle=dashed,offset=2pt,nodesep=1pt}
\ncline{2,1}{3,4}
\ncline{3,4}{2,1}
\caption{The non-special quiver}
\label{fg:non-special_quiver}
\end{minipage}
\end{figure}
\end{example}

\section{The case of cyclic groups}
 \label{sc:cyclic}

We prove the following in this section:

\begin{theorem} \label{th:cyclic}
Let $A$ be a finite small abelian subgroup of $\GL_2(\bC)$ and
%$Y = \AHilb(\bA^2)$
$Y$ be the Hilbert scheme of $A$-orbits
in $\bA^2$.
Then there is an exceptional collection
$
 (E_1, \dots, E_n)
$
in $D^b \coh [\bA^2 / A]$ and a semiorthogonal decomposition
$$
 D^b \coh [\bA^2 / A]
  = \langle E_1, \dots, E_n, \Phi(D^b \coh Y) \rangle,
$$
where $n$ is the number of
indecomposable non-special representations of $G$.
\end{theorem}

%This immediately yields the following:
%
%\begin{corollary}
%%Let $M$ be a smooth Deligne-Mumford stack
%%which contains an open substack $U$
%%isomorphic to $[V / A]$,
%%where $V \subset \bA^2$ is a neighborhood of the origin
%%and $A$ is a cyclic group
%%acting on $V$ through a small linear action on $\bA^2$
%%with isolated fixed point at the origin.
%%Let $N$ be another stack
%%obtained from $M$ by replacing $U$
%%by the minimal resolution $W$ of the quotient scheme $V / A$.
%Let $M$ be a smooth Deligne-Mumford stack
%which is locally isomorphic to $\bA^2/G$ at a point $P$
%where $A$ is a small cyclic group.
%Let $N$ be another stack
%obtained from $M$ by replacing a neighborhood of $P$ by the minimal resolution
%of its coarse moduli space.
%Then one has a semiorthogonal decomposition
%$$
% D^b \coh M = \la E_1, E_2, \dots, E_n, D^b \coh N \ra.
%$$
%\end{corollary}

To prove Theorem \ref{th:cyclic},
we recall Wunram's description of special representations
in the case of cyclic groups.
For relatively prime integers $0<q<n$,
consider the cyclic small subgroup
$
 G = \langle \frac{1}{n}(1, q) \rangle
$
of $\GL_2(\bC)$ generated by
$$
 \frac{1}{n}(1, q)
  = \begin{pmatrix}
     \zeta & 0 \\ 0 & \zeta^q
    \end{pmatrix},
$$
where $\zeta$ is a primitive $n$-th root of unity.
For $a \in \bZ/n\bZ$, let $\rho_a$ denote the irreducible representation of $G$
so that $\rho_a$ sends the above generator to $\zeta^a$.
%We recall how the geometry of the minimal resolution of $\bA^2/G$
%is described by the continued fraction expansion of $n/q$ in this section,
%and collect lemmas which will later be useful.

Define integers $r$, $b_1, \dots, b_r$ and $i_0, \dots, i_{r+1}$ as follows:
Put $i_0 := n$, $i_1:=q$ and define $i_{t+2}, b_{t+1}$ inductively by
$$
i_t = b_{t+1} i_{t+1} - i_{t+2} \quad (0 < i_{t+2} < i_{t+1})
$$
until we finally obtain $i_r=1$ and $i_{r+1}=0$.
This gives a continued fraction expansion
$$
 \frac{n}{q}
  = b_1 - \cfrac{1}{b_2 - \cfrac{1}{\ddots -\cfrac{1}{b_r}}}
$$
and $-b_t$ is the self intersection number
of the $t$-th irreducible exceptional curve $C_t$
in the minimal resolution $Y$ of $\bA^2/G$.

Special representations are described as follows:
\begin{theorem}[{Wunram \cite{Wunram2}}]
Special representations are
$\rho_{i_0} = \rho_{i_{r+1}}, \rho_{i_1}, \dots, \rho_{i_r}$.
\end{theorem}

For an integer $d$ with $0 \le d <n$,
there is a unique expression
\begin{equation}\label{eq:d_t}
d = d_1 i_1 + d_2 i_2 + \dots + d_r i_r
\end{equation}
where $d_i \in \bZ_{\ge 0}$ are non-negative integers satisfying
$$
 0 \le \sum_{t>t_0} d_t i_t < i_{t_0}
$$
for any $t_0$.
\begin{lemma}[{Wunram \cite[Lemma 1]{Wunram2}}]
 \label{lemma:wunramvanishing}
A sequence
$
 (d_1, \dots, d_r) \in (\bZ_{\ge 0})^r
$
is obtained from an integer $d \in [0, n-1]$ as above
if and only if the following hold:
\begin{itemize}
 \item $0 \le d_t \le b_t-1$ for any $t$.
 \item If $d_s= b_s-1$ and $d_t=b_t-1$ for $s<t$,
       then there is $l$ with $s<l<t$ and $d_l \le b_l-3$.
\end{itemize}
\end{lemma}

Let $q' \in [0, n-1]$ be the integer with $qq' \equiv 1 \mod n$.
Then
$\langle \frac{1}{n}(1, q) \rangle$ coincides with
$\langle \frac{1}{n}(q', 1) \rangle$
as a subgroup of $\GL_2(\bC)$.
Introduce the dual sequence $j_0, \dots, j_{r+1}$ by
$j_0 = 0$, $j_1=1$ and
$ j_t = j_{t-1} b_{t-1} - j_{t-2}$
 for $t>1$.
Then one has $j_r = q'$ and $j_{r+1}=n$.

\begin{lemma}[{Wunram \cite[Lemma 2]{Wunram2}}]
Let $d = d_1 i_1 + \dots + d_r i_r$ be as in \eqref{eq:d_t}
and put $f=d_1 j_1 + \dots + d_r j_r$. Then one has $0 \le f \le n-1$ and $qf \equiv d \mod n$.
\end{lemma}

Let $R=\bC[x, y]$ be the coordinate ring of $\bA^2$ and put
$$R_k = R/(x, y^k).$$
For an integer $d\in [0, n-1]$ with $\rho_d$ non-special, take $t$ with
$i_{t-1}> d > i_t$.
Then we define
$$E_d=R_{j_t} \otimes \rho_{d-(j_t-1)q}.$$
Note that the socle of $E_d$ is $\scO_{0} \otimes \rho_d$ and
one has the direct sum decomposition
$
 E_d \cong \bigoplus_{0 \le l < j_t} \rho_{d-l q}
$
as a representation of $G$.
We show that $\{E_d \mid d\text{: non-special}\}$ is a desired exceptional collection
(with respect to the order of $d \in [1, n-1]$).

%%% newcommands %%%
%%%%%%%%%%%%%%
\newcommand{\len}{\preceq}
\newcommand{\lsn}{\prec}
%%%%%%%%%%%%%%
%%%%%%%%%%%%%%

We first show the following:
\begin{proposition} \label{pr:generate}
The following two triangulated subcategories are equal:
$$
 \langle
  \scO_{0} \otimes \rho
 \rangle_{\rho : \text{non-special}}
  = \langle E_d \rangle_{\rho_d :\text{non-special}}.
$$
\end{proposition}
We introduce the following order $\len$ on $\bZ/n\bZ$:
for $a, b \in \bZ/n\bZ$, we write $a \len b$ if $a' \le b'$ holds
for the representatives  $a', b' \in \bZ \cap [0, n-1]$ of $a, b$.
We also write $x \len y$ for $x, y \in \bZ$ if the inequality holds for their classes in $\bZ/n\bZ$.

\begin{lemma}\label{lemma:inequality}
If $0<l<j_t$, then one has $i_{t-1} \len lq$.
\end{lemma}
\begin{proof}
We can write $l = d_1 j_1 + \dots + d_{t-1} j_{t-1}$
as in \eqref{eq:d_t} by using $\{j_t\}$ instead of $\{i_t\}$,
where $(d_1, \dots, d_{t-1}, 0, \dots, 0)$ satisfies the condition in Lemma \ref{lemma:wunramvanishing}.
Then we have $lq \equiv d_1 i_1 + \dots + d_{t-1} i_{t-1} \mod n$.
Since $(d_1, \dots, d_{t-1}, 0, \dots, 0)$ satisfies the condition in Lemma \ref{lemma:wunramvanishing} and is non-zero, 
$d_1 i_1 + \dots + d_{t-1} i_{t-1}$ is an integer in $[i_{t-1}, n-1]$.
This implies the desired inequality.
\end{proof}

Note that the following hold by the definition of $\len$.
\begin{lemma}\label{lemma:obvious}
If $b \ne 0$,
$a+b \len a$ implies $a+b \len b$.
\end{lemma}

\begin{corollary}\label{cor:triangular}
If $i_{t-1}> d > i_t$,
then we have $d \lsn d-lq$ for $0< l < j_t$.
\end{corollary}
\begin{proof}
Since $i_{t-1} \len lq$ by Lemma \ref{lemma:inequality}, we apply Lemma \ref{lemma:obvious} for $a=lq$ and $b=d-lq$
to obtain $d \len d-lq$.
The equality does not hold since $(n, q)=1$.
\end{proof}

\begin{lemma}\label{lemma:non-special}
If $i_{t-1}> d > i_t$, then $\rho_{d-lq}$ is non-special for $0 \le l < j_t$.
\end{lemma}
\begin{proof}
Write $d=d_t i_t + d_{t+1}i_{t+1} + \dots + d_r i_r$ as in \eqref{eq:d_t}
and put $f=d_t j_t + d_{t+1}j_{t+1} + \dots + d_r j_r$.
Then since $\rho_d$ is non-special, we have $f \ge 2j_t$.

Assume that $\rho_{d-lq}$ is special.
Then $d-lq \equiv i_s$ for some $s $ and the above corollary implies $s<t$.
Moreover, $d \equiv i_s + lq$ yields $f \equiv j_s +l$.
On the other hand, since $j_s$ and $l$ are smaller than $j_t$, we see $j_s+l < 2j_t$.
This contradicts  $n> f \ge 2j_t$.
\end{proof}

%Summarizing these facts, we obtain the following result.
\begin{proof}[Proof of Proposition \ref{pr:generate}]
Lemma \ref{lemma:non-special} implies that $E_d$ belongs to
$ \langle
  \scO_{0} \otimes \rho
 \rangle_{\rho : \text{non-special}}$.
Moreover, note that the socle of $E_d$ is $\scO_{0} \otimes \rho_d$.
Then,  for non-special $\rho_f$,  it follows from Corollary \ref{cor:triangular} and the reverse induction on $f$ with respect to $\len$ that $\scO_{0} \otimes \rho_f$ belongs to
$\langle E_d \rangle_{\rho_d :\text{non-special}}$.
\end{proof}

\begin{proposition}\label{pr:exceptional}
$\{E_d\}_{\rho_d:\text{non-special}}$ forms an exceptional collection.
\end{proposition}
\begin{proof}
Take $E_d, E_{d'}$ with $d' \le d$ and suppose $i_{t-1} > d > i_t$ and $i_{t'-1} > d' > i_{t'}$.
To compute $Ext^i(E_d, E_{d'})$,
consider the following projective resolution of $E_d$:
$$
\begin{CD}
0 \to R \otimes \rho_{1+d+q} @>{\begin{pmatrix}{y^{j_t}} \\{-x}\end{pmatrix}}>> R \otimes \rho_{1+d+q-j_tq} \oplus R \otimes \rho_{d+q} @>{\begin{pmatrix}x & y^{j_t}\end{pmatrix}}>> R \otimes \rho_{d+q-j_tq} \to E_d \to 0.
\end{CD}
$$
Then $\bR\Hom_R(E_d, E_{d'})$ splits into the direct sum of
$$
R_{j_{t'}} \otimes \rho_{d'-d+(j_t-j_{t'})q} \overset{\alpha} \to R_{j_{t'}} \otimes \rho_{d'-d-j_{t'}q}
$$
and
$$
R_{j_{t'}} \otimes \rho_{d'-d-1 +(j_t-t_{t'})q} \overset{\beta}\to R_{j_{t'}} \otimes \rho_{d'-d-1 -j_{t'}q}
$$
where $\alpha$ and $\beta$ are the multiplications by $y^{j_t}$.
The degrees of terms of these complexes are determined so that
$\Hom(E_d, E_{d'}) = (\ker \alpha)^G$, 
$\Ext^1(E_d, E_{d'})=(\coker \alpha)^G \oplus (\ker \beta)^G$
and
$\Ext^2(E_d, E_{d'})=(\coker \beta)^G$.

As a representation of $G$, $\ker \alpha$ is the direct sum of
$\rho_{d'-d+lq}$ for $0 \le l < j_t$.
Assume that $\rho_{d'-d+lq}$ is trivial, i.e., $d-d' \equiv lq$.
If $l \ne 0$, then Lemma \ref{lemma:inequality} implies $i_{t-1} \len lq$,
which contradicts $0 \le d' \le d < i_{t-1}$ and $d-d' \equiv lq$.
Therefore, we obtain $l=0$ and $d=d'$.
Thus $(\ker \alpha)^G=0$ if $d \ne d'$ and it is one-dimensional if $d=d'$.
$\coker \alpha$ is the direct sum of $\rho_{d'-d-(j_{t'}-l)q}$ for $0 \le l < j_t$.
Assume $\rho_{d'-d-(j_{t'}-l)q}$ is trivial.
Then we see $d-d'+i_{t'} \equiv lq$, which again contradicts Lemma \ref{lemma:inequality}.
Hence we obtain $(\coker \alpha)^G=0$.
In a similar way, we can show $(\ker \beta)^G=(\coker \beta)^G=0$ and we are done.
\end{proof}

Since $ \langle
  \scO_{0} \otimes \rho
 \rangle_{\rho : \text{non-special}}$
 is the right orthogonal complement of the essential image of $\Phi$,
 Propositions \ref{pr:generate} and \ref{pr:exceptional} imply Theorem \ref{th:cyclic}.

\section{Equivariant McKay correspondence}
 \label{sc:equivariant}
Let $G$ be a finite subgroup of $\GL_2(\bC)$ and
put $G_0 := G \cap \SL(2, \bC)$.
Then $G_0$ is a normal subgroup of $G$
and $A := G / G_0$ is a cyclic group.
There is a natural $G$-action on
$
 Y_0 := \GoHilb \bA^2
$
such that an element $g \in G$
sends a subschema $Z \in \GoHilb \bA^2$
to its image $g \cdot Z$
by the action $g : \bA^2 \to \bA^2$.
Since $Z$ is $G_0$-invariant
by the definition of $\GoHilb \bA^2$,
this $G$-action on $Y_0$ descends
to an $A = G / G_0$-action
on the scheme $Y_0$.

\begin{theorem} \label{th:equivariant}
There is a derived equivalence
$$
 \Phi_0 : D^b \coh [Y_0 / A] \simto D^b \coh [\bA^2 / G].
$$
\end{theorem}

\begin{proof}
The groups $G$ and $A$ acts naturally on $Y_0$,
and there is a natural morphism
\begin{align*}
 \varphi : [Y_0/G] \to [Y_0/A]
\end{align*}
coming from the surjection
$
 G \twoheadrightarrow A
$
The push-forward functor
$$
 \varphi_* : D^b \coh [Y_0 / G] \to D^b \coh [Y_0 / A]
$$
sends a $G$-equivariant cohenrent sheaf $\scE$ on $Y_0$
to the $G_0$-invariant subsheaf $\scE^{G_0}$
equipped with the natural $A$-equivariant structure.
The pull-back functor
$$
 \varphi^* :
  D^b \coh [Y_0 / A] \to D^b \coh [Y_0 / G]
$$
sends an $A$-equivariant coherent sheaf on $Y_0$
to the same sheaf
considered as a $G$-equivariant coherent sheaf
through the surjective homomorphism
$
 G \twoheadrightarrow A.
$

Consider the diagram
$$
\begin{psmatrix}[rowsep=5mm]
 & \scZ & \\[-7mm]
 & \rotatebox{270}{$\subset$} \\[-3mm]
 & Y_0 \times \bA^2 & \\[6mm]
 Y_0 & & \bA^2 
\end{psmatrix}
\psset{arrows=->,shortput=nab,nodesep=4pt}
\ncline{3,2}{4,1}_{\pi_{Y_0}}
\ncline{3,2}{4,3}^{\pi_{\bA^2}}
$$
where
$
 \scZ \subset Y_0 \times \bA^2
$
is the universal subscheme
and
$\pi_{\bA^2}$ and $\pi_{Y_0}$ are the natural projections.
By taking the quotient of the whole diagram
with respect to the action of $G$,
one obtains another  diagram
$$
\begin{psmatrix}[rowsep=5mm]
 & \big[ \scZ / G \big] & \\[-7mm]
 & \rotatebox{270}{$\subset$} \\[-3mm]
 & \big[ Y_0 \times \bA^2/G \big] & \\[6mm]
 \big[ Y_0 / G \big] & & \big[ \bA^2 / G \big].
\end{psmatrix}
\psset{arrows=->,shortput=nab,nodesep=4pt}
\ncline{3,2}{4,1}_{\pi_{[Y_0/G]}}
\ncline{3,2}{4,3}^{\pi_{[\bA^2/G]}}
$$
Then we can define an integral functor
$$
 \Phi_0 : D^b \coh [Y_0/A] \to D^b \coh [\bA^2/G]
$$
by
$$
 \Phi_0 (-) = \pi_{[\bA^2/G]*}(\scO_{[\scZ/G]} \otimes
  \pi_{[Y_0/G]}^*(\varphi^*(-))),
$$
and another functor
$$
 \Psi_0 : D^b \coh [\bA^2/G] \to D^b \coh [Y_0/A]
$$
by
$$
 \Psi_0(-) =
  \varphi_* ( \pi_{[Y_0/G]*}
   (\scO_{[\scZ/G]}^{\vee}[2]
  \otimes \det \rhonat \otimes \pi_{[\bA^2/G]}^*(-))),
$$
where
$$
 \scO_{[\scZ/G]}^\vee =
  {\mathop{\bR \scH om}\nolimits}_{
   \scO_{[Y_0 \times \bA^2/G]}}
   (\scO_{[\scZ/G]}, \scO_{[Y_0 \times \bA^2/G]}).
$$
The functor $\Psi_0$ is both left and right adjoint to $\Phi_0$ since
\begin{itemize}
 \item
the functor $\pi_{[\bA^2/G]*}$ is
right adjoint to ${\pi_{[\bA^2/G]}^*}$
and left adjoint to
$$
 \pi_{[\bA^2/G]}^{!}(-)
  = \pi_{[\bA^2/G]}^*(-) \otimes
   \pi_{[Y/G]}^*(\omega_{[Y_G]})[2]
  = \pi_{[\bA^2/G]}^*(-) \otimes \det \rhonat[2],
$$
 \item
the functor $\pi_{[Y_0/G]*}$ is right adjoint to
$\pi_{[Y_0/G]}^*$
and left adjoint to
$$
 \pi_{[Y_0/G]}^!(-)
  = \pi_{[Y_0/G]}^*(-) \otimes
   \pi_{[\bA^2/G]}^*(\omega_{[\bA^2/G]})[2]
  = \pi_{[Y_0/G]}^*(-) \otimes \det \rhonat [2],
$$
 \item
the functor $- \otimes \scO_{[\scZ/G]}$ is
both left and right adjoint
to $- \otimes \scO_{[\scZ/G]}^\vee$, and
 \item
the functor $\varphi_*$ is both left and right adjoint
to $\varphi^*$.
\end{itemize}

By restricting $G$-actions to $G_0$-actions
and forgetting $A$-actions,
we can also define the functor
$
 \Phi_0': D^b \coh Y_0 \to D^b \coh [\bA^2/G_0]
$
and its adjoint $\Psi_0'$ in the same way as above,
which are equivalences by
\cite{Kapranov-Vasserot,
Bridgeland-King-Reid}.

Let $\alpha$ be any object of $D^b \coh [Y_0/A]$ and
consider the adjunction morphism
$
 \nu: \alpha \to \Psi_0 \Phi_0 (\alpha).
$
If we send the morphism $\nu$
by the pull-back functor
$$
 \varphi_{A}^* : D^b \coh [Y_0/A] \to D^b \coh Y_0
$$
along the morphism
$\varphi_{A} : Y_0 \to [Y_0/A]$,
then the resulting morphism
$\varphi_{A}^*(\nu)$ is an isomorphism
in $D^b \coh Y_0$
since $\Phi_0'$ and $\Psi_0'$ are equivalences.
Although the functor $\varphi_{A}^*$ is not full,
it is faithful and
this shows that the morphisms $\nu$ is an isomorphism.
We can also show that the adjunction morphism
$\Phi_0 \Psi_0 (\beta) \to \beta$ is an isomorphism
for any object $\beta$ of $D^b \coh [\bA^2/G]$
in the same way,
so that $\Phi_0$ and $\Psi_0$ are equivalences.
\end{proof}

%\section{Root constructions and semiorthogonal decompositions}
% \label{sc:root}
%
%% We say that a stack $M$ does not have a generic stabilizer
%% if $M$ has an open dense subset isomorphic to a scheme.
%% We say that $M$ has a generic stabilizer $G$
%% if $M$ is the quotient $[N / G]$ of another stack $N$
%% without generic stabilizer
%% by the trivial action of $G$.
%In this section,
%we show that a smooth Deligne-Mumford stack
%containing divisors with non-trivial stabilizers
%can be replaced with another stack without such divisors
%by removing a semiorthogonal summand
%from its derived category of coherent sheaves.
%This operation is inverse to the {\it root construction}
%with respect to a line bundle with a section,
%defined independently in \cite{Abramovich-Graber-Vistoli}
%and \cite{Cadman_US}.

\section{The root stack of a line bundle}
 \label{sc:root1}

For a line bundle $\scL$ on a Deligne-Mumford stack $\scX$
and a positive integer $r$,
the $r$-the root of $\scL$
is the stack $\pi : \sqrt[r]{\scL/\scX} \to \scX$ over $\scX$
such that
%is defined as the quotient stack
%$
% [L / \bGm]
%$
%with respect to the action of the $\bGm$-action
%defined by $x \mapsto \alpha^{-r} \cdot x$
%for $x \in L$ and $\alpha \in \bGm$.
\begin{itemize}
 \item
an object
over a scheme $T$
%over a scheme $\varphi : T \to \scX$ over $\scX$
is a triple $(\varphi, \scM, \phi)$
consisting of a morphism
$\varphi : T \to \scX$ of stacks,
a line bundle $\scM$ on $T$, and
an isomorphism
$
 \phi : \scM^{\otimes r} \simto \varphi^* \scL
$
of line bundles on $T$, and
 \item
a morphism is a commutative diagram
$$
\begin{psmatrix}[colsep=1,rowsep=1]
 T & & T' \\
  & \scX
\end{psmatrix}
\psset{arrows=->,shortput=nab,nodesep=3pt}
\ncline{1,1}{1,3}^{\varphi''}
\ncline{1,1}{2,2}_{\varphi}
\ncline{1,3}{2,2}^{\varphi'}
$$
and an isomorphism
$
 \phi'' : \scM^{\otimes r} \simto {\varphi''}^* {\scM'}^{\otimes r}
$
making the diagram\\[3mm]
$$
\begin{psmatrix}[colsep=1,rowsep=1]
 \scM^{\otimes r} & & {\varphi''}^{*} {\scM'}^{\otimes r} \\
  & \rnode{A}{\varphi^* \scL} \cong
  \rnode{B}{(\varphi' \circ \varphi'')^* \scL}
\end{psmatrix}
\psset{arrows=->,shortput=nab,nodesep=3pt}
\ncline{1,1}{1,3}^{\phi''}
\ncline{1,1}{A}_{\phi}
\ncline{1,3}{B}^[npos=.3]{{\varphi''}^*(\phi')}
$$
commute.
\end{itemize}

Let $(\scM, \Phi)$ be the universal object
on $\sqrt[r]{\scL/\scX}$,
so that $\scM$ is a line bundle on $\sqrt[r]{\scL/\scD}$
and $\Phi : \scM^{\otimes r} \to \pi^*{\scL}$ is an isomorphism
of line bundles.

The structure morphism
$
 \pi : \sqrt[r]{\scL/\scX} \to \scX
$
makes the root stack $\sqrt[r]{\scL/\scX}$
into an essentially trivial gerb
over $\scX$ banded by $\mu_r$,
where $\mu_r$ is the kernel
of the $r$-th power map
$\bGm \to \bGm$
between the multiplicative groups.
This means that $\sqrt[r]{\scL/\scX}$
is the $[pt/\mu_r]$-bundle
associated with the principal $\bGm$-bundle
$L := \scL \setminus (\text{the zero section})$.

Now we prove Theorem \ref{th:root1}:

\begin{proof}[Proof of Theorem \ref{th:root1}]
For any coherent sheaf $\scF$ on $\sqrt[r]{\scL/\scX}$
and any integer $i$,
one has the adjunction morphism
$
 \pi^* \pi_* (\scF \otimes \scM^{\otimes i})
  \to \scF \otimes \scM^{\otimes i},
$
whose direct sum gives the morphism
\begin{align} \label{eq:adjunction1}
 \bigoplus_{i=0}^{r-1}
  \pi^*(\pi_* (\scF \otimes \scM^{\otimes(-i)}))
   \otimes \scM^{\otimes i}
 \to \scF.
\end{align}
Since this is a morphism of sheaves,
one can work locally
to show that it is an isomorphism.
Take an open set $\scU \subset \scX$
where the line bundle $\scL$ is trivial,
so that the root stack is the trivial gerb
given by the direct product
$\scU \times [pt / \mu_r]$
with the classifying stack.
Then the sheaf $\scM^{\otimes i}|_{\scU}$
corresponds to $\scO_\scU \otimes \rho_i$
under the equivalence
$
 \coh (\scU \times [pt / \mu_r])
  \cong (\coh \scU) \otimes (\rep \mu_r),
$
where $\rep \mu_r$ is the category of finite-dimensional
representations of $\mu_r$
and $\rho_i$ is the representation
sending $\alpha \in \mu_r$ to $\alpha^i \in \bGm$.
This immediately shows
that \eqref{eq:adjunction1} is an isomorphism.
The same local consideration also shows
that $(\pi^* \coh \scX) \otimes \scM^{\otimes i}$
for $i=0, \dots, r-1$ are mutually orthogonal,
and Theorem \ref{th:root1} is proved.
\end{proof}

\section{The root stack of a line bundle with a section}
\label{sc:root2}

Let $(\scL, \sigma)$ be a pair
of a line bundle $\scL \to \scX$
and a section
$
 \sigma : \scX \to \scL.
$
The stack $\sqrt[r]{(\scL, \sigma)/\scX}$
of the $r$-th roots of $(\scL, \sigma)$ is the stack
such that
\begin{itemize}
 \item
an object over $T$
is a quadruple $(\varphi, \scM, \phi, \tau)$
consisting of an object $(\varphi, \scM, \phi)$
of $\sqrt[r]{\scL/\scX}$ over $T$ and
a section $\tau$ of $\scM$ such that $\phi(\tau^{\otimes r}) = \varphi^*\sigma$,
and
 \item
a morphism is a morphism $(\varphi'', \phi'')$
of $\sqrt[r]{\scL/\scX}$ such that $\phi''(\tau) = \tau'$.
%such that $\varphi commutative diagram
%$$
%\begin{psmatrix}[colsep=1,rowsep=1]
% T & & T' \\
%  & \scX
%\end{psmatrix}
%\psset{arrows=->,shortput=nab,nodesep=3pt}
%\ncline{1,1}{1,3}^{\varphi''}
%\ncline{1,1}{2,2}_{\varphi}
%\ncline{1,3}{2,2}^{\varphi'}
%$$
%and an isomorphism
%$
% \phi'' : \scM^{\otimes r} \simto {\varphi''}^* {\scM'}^{\otimes r}
%$
%making the diagram\\[3mm]
%$$
%\begin{psmatrix}[colsep=1,rowsep=1]
% \scM^{\otimes r} & & {\varphi''}^{*} {\scM'}^{\otimes r} \\
%  & \rnode{A}{\varphi^* \scL} \cong
%  \rnode{B}{(\varphi' \circ \varphi'')^* \scL}
%\end{psmatrix}
%\psset{arrows=->,shortput=nab,nodesep=3pt}
%\ncline{1,1}{1,3}^{\phi''}
%\ncline{1,1}{A}_{\phi}
%\ncline{1,3}{B}^[npos=.3]{{\varphi''}^*(\phi')}
%$$
%commute.
\end{itemize}
%If $\scY$ is the zero locus of $\sigma$,
%then the restriction of $\sqrt[r]{(\scL, \sigma)/\scX}$
%to $\scX \setminus \scY$ is isomorphic to $\scX \setminus \scY$,
%and the restriction of $\sqrt[r]{(\scL, \sigma)/\scX}$
%to $\scY$ is the $r$-th infinitesimal neighborhood
%of $\sqrt[r]{\scL|_\scY/\scY}$ in its universal line bundle.

Assume that $\scX$ is a smooth Deligne-Mumford stack and
$
 \jbar : \scD \to \scX
$
is a closed embedding of a smooth divisor.
The canonical section of the line bundle $\scO(\scD)$
associated with the divisor $\scD$ will be denoted by
$1 \in \Gamma(\scO(\scD))$.
Let
\begin{align} \label{eq:embedding1}
 \left. \sqrt[r]{(\scO(\scD), 1)/\scX} \right|_\scD
  \subset
 \sqrt[r]{(\scO(\scD), 1)/\scX}
\end{align}
be the substack
consisting of objects $(\varphi, \scM, \phi)$
such that the morphism $\varphi : T \to \scX$
factors through $\jbar$.
There is a closed embedding
$$
 \sqrt[r]{\scO_{\scD}(\scD)/\scD}
  \hookrightarrow
 \left. \sqrt[r]{(\scO(\scD), 1)/\scX} \right|_\scD
$$
sending an $r$-th root $\scM$ of $\scO_\scD(\scD)$
to the same $\scM$ together with the zero section.
The composition of this morphism
with the embedding \eqref{eq:embedding1}
%$
% \left. \sqrt[r]{(\scO(\scD), 1)/\scX} \right|_\scD
%  \to \sqrt[r]{(\scO(\scD), 1)/\scX}
%$
will be denoted by $j$,
which fits into the commutative diagram
%$$
% j : \sqrt[r]{\scO_{\scD}(\scD)/\scD}
%  \to \sqrt[r]{(\scO(\scD), 1)/\scX}.
%$$
$$
\begin{CD}
 \sqrt[r]{\scO_{\scD}(\scD)/\scD}
  @>{j}>> \sqrt[r]{(\scO(\scD), 1)/\scX} \\
  @V{\pi_\scD}VV @VV{\pi_\scX}V \\
 \scD @>{\jbar}>> \scX.
\end{CD}
$$
%Let
%$
% \pi_{\scX}: \sqrt[r]{(\scO(\scD), 1)/\scX} \to \scX
%$
%and
%$
% \pi_{\scD}: \sqrt[r]{\scO_{\scD}(\scD)/\scD} \to \scD
%$
%be the natural projections.
The universal line bundle on $\sqrt[r]{(\scO(\scD), 1)/\scX}$
will be denoted by $\scM$.

The following proposition gives
Theorem \ref{th:root2}:

\begin{proposition} \label{pr:root2}
\ 
\begin{itemize}
\item[{\rm (i)}]
The functor
$
 j_*\pi_{\scD}^* :
  D^b(\coh \scD) \to D^b(\coh \sqrt[r]{(\scO(\scD), 1)/\scX})
$
is fully faithful if $r>1$.
%If $\scE$ is an exceptional object in $D^b(\coh \scD)$, then
%$j_*\pi_D^*(\scE)$ is an exceptional object in $D^b(\sqrt[r]{(\scO(\scD), 1)/\scX})$.
\item[{\rm (ii)}]
%If $D^b(\coh \scD)$ has a full exceptional collection $\scE_1, \dots, \scE_t$, then
One has a semiorthogonal decomposition
\begin{multline*}
 D^b \coh \sqrt[r]{(\scO(\scD), 1)/\scX} =
 \langle
   j_*\pi_{\scD}^* D^b(\coh \scD) \otimes \scM^{\otimes r-1},
%   j_*\pi_{\scD}^* D^b(\coh \scD) \otimes \scM^{\otimes r-2}, \\
    \ldots, \\
   j_*\pi_{\scD}^* D^b(\coh \scD) \otimes \scM,
  \pi_{\scX}^* D^b \coh \scX
  \rangle.
\end{multline*}
\end{itemize}
\end{proposition}

\begin{proof}
(i)
For any objects $\alpha$ and $\beta$ of $D^b(\coh \scD)$
and any $q \in \bZ$,
we show that the natural morphism
\begin{equation} \label{eqn:map}
 \Hom^q(\alpha, \beta)
  \to \Hom^q(j_*\pi_{\scD}^* \alpha, j_*\pi_{\scD}^*\beta)
  \cong \Hom^q(j^*j_*\pi_{\scD}^* \alpha, \pi_{\scD}^*\beta)
\end{equation}
is an isomorphism.
We may assume that $\alpha$ and $\beta$ are sheaves.
Then we have
\begin{equation} \label{eqn:j*j*}
 H^i(j^*j_*\pi_{\scD}^* \alpha)
  \cong
   \begin{cases}
 \pi_{\scD}^* \alpha & i = 0, \\
 \pi_{\scD}^* \alpha \otimes \scM^{-1} & i = -1, \\
 0 & \text{otherwise}.
   \end{cases}
\end{equation}
If $r>1$, Theorem \ref{th:root1} shows% that the group
$$
 \Hom^q(\pi_{\scD}^* \alpha \otimes \scM^{-1}, \pi_{\scD}^*\beta)
  \cong 0
$$
%vanishes
and
$$
 \Hom^q(\pi_{\scD}^* \alpha, \pi_{\scD}^*\beta)
 \cong \Hom^q(\alpha, \beta)
$$
for any $q$.
It follows that the spectral sequence
$$
 \Hom^p(H^{-q}(j^* j_* \pi_\scD^* \alpha), \pi_\scD^* \beta)
  \Rightarrow
 \Hom^{p+q}(j^* j_* \pi_\scD^* \alpha, \pi_\scD^* \beta)
$$
degenerates
and \eqref{eqn:map} is an isomorphism.

(ii)
The subcategory
$\pi_{\scX}^*(D^b(\coh \scX))$ is admissible
since the functor $\pi_{\scX}^*$ has both right and left adjoints.
The subcategories
$
 j_*\pi_{\scD}^* D^b(\coh \scD) \otimes \scM^{\otimes i}
$
are also admissible
since the functor $j_*\pi_{\scD}^*$ has
both left and right adjoints
and the functor $(-) \otimes \scM^{\otimes i}$
is an equivalence.

We can deduce that
$
 j_*\pi_{\scD}^* D^b(\coh \scD) \otimes \scM^{\otimes i}
$
are right orthogonal to $\pi_{\scX}^*(D^b(\coh X))$
for $1 \le i \le r-1$ from
\begin{align*}
 \Hom(\pi_{\scX}^* \alpha,
       j_*(\pi_{\scD}^* \beta \otimes \scM^{\otimes i})) 
  &\cong \Hom(j^*\pi_{\scX}^*\alpha,
               \pi_{\scD}^* \beta \otimes \scM^{\otimes i}) \\
  &\cong \Hom(\pi_{\scD}^* {\jbar}^* \alpha,
               \pi_{\scD}^* \beta \otimes \scM^{\otimes i}) \\
  &=0,
\end{align*}
where $\jbar : \scD \to \scX$ is the closed immersion.
Similarly, \eqref{eqn:j*j*} implies
$$
 \Hom(j_*\pi_{\scD}^* \alpha \otimes \scM^{\otimes k},
       j_*\pi_{\scD}^* \beta \otimes \scM^{\otimes l})
 = 0
$$
for $1 \le k < l \le r-1$.

It remains to show that
any object $\scE$ of $D^b \coh \sqrt[r]{(\scO(D), 1)/\scX}$
is obtained from objects of
$
 j_*\pi_{\scD}^* (D^b \coh \scD) \otimes \scM^{\otimes i}
$
for $1 \le i \le r-1$ and $\pi_{\scX}^* D^b \coh \scX$
by taking shifts and cones.
Since $\pi_{\scX}$ is an isomorphism outside $\scD$,
the mapping cone
$\Cone(\pi_{\scX}^*{\pi_{\scX}}_* \scE \to \scE)$
of the adjunction morphism is supported
on $\sqrt[r]{\scO_{\scD}(\scD)/\scD}$.
It follows that $\scE$ can be obtained from
$\pi_{\scX}^*{\pi_{\scX}}_* \scE$ and
an object supported on $\sqrt[r]{\scO_{\scD}(\scD)/\scD}$
by taking cones.

An object supported on $\sqrt[r]{\scO_{\scD}(\scD)/\scD}$
is obtained from objects of
$j_* D^b \coh \sqrt[r]{\scO_{\scD}(\scD)/\scD}$
by taking cones,
which in turn can be obtained from objects of 
$
 j_*\pi_{\scD}^* D^b(\coh \scD) \otimes \scM^{\otimes i}
$
for $0 \le i \le r-1$ by Theorem \ref{th:root1}.
Finally, we have to show that an object of 
$ j_*\pi_{\scD}^* D^b(\coh \scD)$
 is obtained from objects of $\pi_{\scX}^* D^b \coh \scX$ and
$
 j_*\pi_{\scD}^* D^b(\coh \scD) \otimes \scM^{\otimes i}
$
for $1 \le i \le r-1$.
If $\alpha$ is a sheaf in $D^b(\coh \scD)$, then
$\pi_{\scX}^*{\bar j}_* \alpha$ has a filtration whose factors
are $j_*\pi_{\scD}^* \alpha \otimes \scM^{\otimes i}$ for $0 \le i \le r-1$.
Thus $j_*\pi_{\scD}^* \alpha$ is obtained from $\pi_{\scX}^*{\bar j}_* \alpha$
and $j_*\pi_{\scD}^* \alpha \otimes \scM^{\otimes i}$ for $1 \le i \le r-1$
by taking shifts and cones.
This concludes the proof of Proposition \ref{pr:root2}.
\end{proof}

\begin{corollary} \label{cor:root}
If both $\scX$ and $\scD$ have full exceptional collections,
then so does the root stack
$\sqrt[r]{(\scO(\scD), 1)/\scX}$.
\end{corollary}

\section{Iterations of root constructions}
 \label{sc:iteration}

A smooth Deligne-Mumford stack
$\scY$ is said to be {\em canonical}
if the locus where
the structure morphism $\scY \to Y$
to the coarse moduli space $Y$
is not an isomorphism
has codimension greater than one
\cite[Definition 4.4]{Fantechi-Mann-Nironi}.
The canonical stack has the universal property
\cite[Theorem 4.6]{Fantechi-Mann-Nironi}
that any dominant codimension-preserving
morphism $f : \scX \to Y$
from a smooth stack $\scX$ without generic stabilizers
factors through $\scY \to Y$ uniquely
up to unique 2-arrow;\\[0mm]
$$
\begin{psmatrix}[colsep=1.5,rowsep=1]
 \scX & \scY \\
  & Y.
\end{psmatrix}
\psset{arrows=->,nodesep=3pt,shortput=nab}
\ncline[linestyle=dashed]{1,1}{1,2}^{\exists ! g}
\ncline{1,2}{2,2}^{\epsilon}
\ncline{1,1}{2,2}_{f}
$$
For a variety $X$ with at worst quotient singularities,
there is a canonical stack $\scX^\can$
whose coarse moduli space is isomorphic to $X$,
which is determined uniquely
%by this property
up to isomorphism
\cite[Remark 4.9]{Fantechi-Mann-Nironi}.
%
%Let $X$ be a variety with at worst quotient singularity and
%$\scX^{\can}$ be its canonical stack.
%Any smooth Deligne-Mumford stack $\scX$
%with trivial generic stabilizer
%whose coarse moduli space is isomorphic to $X$
%can be obtained from $\scX^\can$
%by iterated root constructions
For effective divisors $\scD_1, \ldots, \scD_s$ on $\scX^{\can}$
and positive integers $r_1, \dots, r_s$,
the fiber product
\begin{align} \label{eq:iteration3}
 \scX =
  \sqrt[r_1]{(\scO(\scD_1),1)/\scX^{\can}}
   \times_{\scX^{\can}} \dots \times_{\scX^{\can}}
    \sqrt[r_s]{(\scO(\scD_s),1)/\scX^{\can}}
  \xto{\quad \varphi \quad} \scX^\can
\end{align}
%along prime divisors $\scD_1, \ldots, \scD_s$ on $\scX^{\can}$
%for some positive integers $r_1, \dots, r_s$
is obtained by iterating root constructions
\cite[Remark 2.2.5]{Cadman_US}.
%If $\sum_i \scD_i$ is a simple normal crossing divisor,
%then $\scX$ is also a smooth Deligne-Mumford stack.
If we write the reduced closed substack
$(\varphi^{-1}(\scD_i))_{\mathrm{red}}$
as $\scDtilde_i$,
then one has $\varphi^* \scD_i = r_i \scDtilde_i$.
The numbers $(r_1, \ldots, r_s)$ are called
the {\em divisor multiplicities} of $\scX$
\cite[Remark 3.7]{Fantechi-Mann-Nironi}.
If each $\scD_i$ is smooth and
$\sum_i \scD_i$ is a simple normal crossing divisor,
then $\scX$ and $\scDtilde_i$ are smooth
and $\sum_i \scDtilde_i$
is a simple normal crossing divisor
\cite[Section 1.3.b]{Fantechi-Mann-Nironi}.

The following lemma can be proved
in just the same way as (2) of
\cite[Theorem 5.2]{Fantechi-Mann-Nironi}.
We give a proof for the reader's convinience.

\begin{proposition}[{cf.~\cite[Theorem 5.2]{Fantechi-Mann-Nironi}}]
 \label{pr:iteration}
Let $X$ be a variety over $\bC$ with at worst quotient singularities,
$\scX^\can$ be its canonical stack,
and $\sum_i \scD_i$ be an effective
simple normal crossing divisor on $\scX^\can$
such that each irreducible component $\scD_i$ is smooth.
Let further
$(r_1, \ldots, r_s)$ be a sequence of positive integers
and $\scX$ be the smooth Deligne-Mumford stack
obtained by iterated root constructions
as in \eqref{eq:iteration3}.
Then $\scX$ is characterized
by the following properties
up to isomorphism:
\begin{itemize}
 \item
$\scX$ is a smooth separated Deligne-Mumford stack
without generic stabilizers.
 \item
$\scX$ has the same coarse moduli space
as $\scX^{\can}$.
 \item
The canonical morphism
$
 \varphi : \scX \to \scX^\can
$
is an isomorphism outside $\bigcup_i \scD_i$.
 \item
The pull-back of $\scD_i$ is $r_i$ times a prime divisor.
\end{itemize}
\end{proposition}
\begin{proof}

Let $\varphi:\scX \to \scX^\can$ be the stack obtained as in \eqref{eq:iteration3}.
Then it follows from \cite[Section 2.1]{Bayer-Cadman} that $\scX$ has the four properties
in the statement.

Conversely, 
suppose $f:\scX' \to \scX^\can$ is an arbitrary smooth Deligne-Mumford stack with the above properties.
Then, by the universal property of the root stack, one has a morphism
$g: \scX' \to \scX$ with $\varphi\circ g=f$.
Since $\scX'$ has the same coarse moduli space as $\scX^{\can}$,
$g$ is a surjective morphism.
Let $S \to \scX$ be an \'etale atlas and let
$\scY$ be the fiber product of $S$ and $\scX'$ over $\scX$
with the induced morphism $\gtilde: \scY \to S$:
\begin{align}
\begin{psmatrix}[colsep=15mm,rowsep=5mm]
 \cY & \cX' \\
 & & \cX^\can \\
 S & \cX
\end{psmatrix}
\psset{arrows=->,shortput=nab,nodesep=5pt}
\ncline{1,1}{1,2}
\ncline{1,2}{2,3}^{f}
\ncline{3,1}{3,2}
\ncline{3,2}{2,3}_{\varphi}
\ncline{1,1}{3,1}_{\gtilde}
\ncline{1,2}{3,2}^{g}
\end{align}

\begin{step}
$\gtilde$ is \'etale.
\end{step}
Let $U \to \scY$ be an \'etale atlas.
The morphism $U \to S$ is flat
since $U$ and $S$ are smooth and $U \to S$ has $0$-dimensional fibers.
It is also \'etale over $\scX^\can \setminus \bigcup_i \scD_i$ by our assumption.
Moreover, the pull-back of a prime divisor of $S$ to $U$ is a reduced divisor
and thus $U \to S$ is unramified in codimension one.
Therefore $U \to S$ is \'etale in codimension one.
Then it must be \'etale by the purity of the branch locus.
This implies that $\gtilde$ is \'etale.

\begin{step}
$\scY$ is an algebraic space.
\end{step}

The \'etale morphism
$\gtilde :\scY \to S$ factors through the coarse moduli space $Y$ of $\scY$ 
by our assumption that $S$ is a scheme.
This implies that the the morphism from $\scY$ to $Y$ is unramified.
Take an \'etale covering $\{Y_{\alpha} \to Y\}$ of $Y$
such that $\scY \times_Y Y_\alpha$ is isomorphic to the quotient stack
$[U_\alpha/\Gamma_\alpha]$,
where $U_\alpha$ is a scheme for each $\alpha$
and $\Gamma_\alpha$ is a finite group acting on $U_\alpha$
\cite[Lemma 2.2.3]{MR1862797}.
By the unramifiedness of the morphism $\scY \to Y$,
the quotient morphism $U_\alpha \to U_\alpha/\Gamma_\alpha$ is also unramified
and hence the action of $\Gamma_\alpha$ on $U_\alpha$ 
is free.
(Suppose a subgroup $H \subset \Gamma_\alpha$ fixes a closed point $P$ of $U_\alpha$.
Then the quotient morphism $h:U_\alpha \to U_\alpha/H$ is unramified
and therefore $H$ acts trivially on the tangent space of $P$.
This implies that $H$ acts trivially on a neighbourhood of $P$.
Since $\scY$ doesn't have a generic stabilizer, $H$ must be trivial.)
This implies that  $\scY=Y$ is an algebraic space.

\begin{step}
$\gtilde$ is an isomorphism.
\end{step}
Since $\gtilde$ is an \'etale surjective separeted morphism of algebraic spaces,
the diagonal morphism $\Delta_{\gtilde}$ of $\gtilde$ is an open and closed immersion.
Moreover, $\gtilde$ is an isomorphism over an open dense subset of $S$
and therefore $\Delta_{\gtilde}$ is actually an isomorphism.
This means that $\gtilde$ becomes an isomorphism if we take a base change by $\gtilde$ itself.
Since $\gtilde$ is \'etale surjective,  $\gtilde$ is actually an isomoprhism.
This proves that $g$ is an isomorphism.
\end{proof}

The following is the main result in this section:

\begin{proposition} \label{pr:root}
Let $\scX$ be a two-dimensional
smooth separated Deligne-Mumford stack
without generic stabilizers.
Assume that
\begin{itemize}
 \item
the canonical morphism
$
 \varphi : \scX \to \scX^\can
$
to the canonical stack $\scX^\can$
of the coarse moduli space $X$
is an isomorphism
outside a simple normal crossing divisor
$\sum_i \scD_i$ on $\scX^\can$,
 \item
the pull-back $\varphi^* \scD_i$ is a multiple
of a prime divisor, and
 \item
each irreducible component $\scD_i$
is a smooth rational stack.
\end{itemize}
Then there exists an exceptional collection
$(E_1, \dots, E_{\ell})$
%in $D^b(\scX)$
and a semiorthogonal decomposition
\begin{equation} \label{eq:div}
 D^b \coh \scX =
  \langle
   E_1, \dots, E_{\ell}, \varphi^* D^b \coh \scX^{\can}
  \rangle.
\end{equation}
\end{proposition}
\begin{proof}
%(Old version)
%Put
%$$
% \scX_1 =
%  \sqrt[r_2]{(\scO(\scD_2), 1)/\scX^{\can}}
%   \times_{\scX^{\can}} \dots \times_{\scX^{\can}}
%    \sqrt[r_s]{(\scO(\scD_s),1)/\scX^{\can}}
%$$
%and let $\scD_1' \subset \scX_1$ and
%$\widetilde{\scD}_1 \subset \scX$
%be the prime divisors corresponding to $D_1$.
%Then $\scX$ is isomorphic to
%$
% \sqrt[r_1]{(\scO(\scD_1'), 1)/\scX_1}
%$
%and $\widetilde{\scD}_1$ is isomorphic to
%$
% \sqrt[r_1]{\scO_{\scD_1'}(\scD_1')/\scD_1'}.
%$
%Since $D_1 \cong \bP^1$ has a full exceptional collection
%\cite{Beilinson},
%so does $\scD_1'$ by Corollary \ref{cor:root} and
%$\widetilde{\scD}_1$ by Lemma \ref{lemma:orthogonal}.
%Then Lemma \ref{lemma:root} yields a semiorthogonal decomposition
%of $D^b(\scX)$ by exceptional objects and $D^b(\scX_1)$.
%By induction on $s$, we obtain the assertion.
%\end{proof}
%
%\begin{proof}
%(New version)
Put
$$
 \scX_1 =
  \sqrt[r_2]{(\scO(\scD_2), 1)/\scX^{\can}}
   \times_{\scX^{\can}} \dots \times_{\scX^{\can}}
    \sqrt[r_s]{(\scO(\scD_s),1)/\scX^{\can}}
$$
and let
$
 \scD \subset \scX_1
$
be the prime divisor corresponding to $D_1$.
Then $\scX$ is isomorphic to
$\sqrt[r_1]{(\scO(\scD), 1)/\scX_1}$
and one has a semiorthogonal decomposition
\begin{align*}
 D^b \coh \scX =
 \langle
   j_*\pi_{\scD}^* D^b(\coh \scD) \otimes \scM^{\otimes r_1-1},
%   j_*\pi_{\scD}^* D^b(\coh \scD) \otimes \scM^{\otimes r-2}, \\
    \ldots, %\\
   j_*\pi_{\scD}^* D^b(\coh \scD) \otimes \scM,
  \varphi^* D^b \coh \scX_1
  \rangle
\end{align*}
by Proposition \ref{pr:root2}.
Since $D_1$ is a smooth rational curve
and $\sum_i \scD_i$ is a simple normal crossing divisor,
the divisor $\scD$ is smooth
and its coarse moduli space is a smooth rational curve.
It follows that $\scD$ is isomorphic
to a weighted projective line in the sense
of Geigle and Lenzing \cite{Geigle-Lenzing_WPC},
so that the derived category $D^b \coh \scD$
and hence the right orthogonal
to $\varphi^* D^b \coh \scX_1$
in $D^b \coh \scX$ has
a full exceptional collection.
Now the assertion follows from induction on $s$.
\end{proof}

We end this section with the following lemma:

\begin{lemma} \label{lm:root1}
The stack $[Y_0/A]$ satisfies the assumption
of Proposition \ref{pr:root}.
\end{lemma}

\begin{proof}
$Y_0$ is isomorphic to the minimal resolution
of $\bA^2 / G_0$,
and let $E = \bigcup_i D_i$ be the exceptional divisor.
The canonical morphism is clearly an isomorphism outside $E$,
which is a normal crossing divisor.
Since $A$ is a cyclic group
whose action on $Y_0$ has no generic stabilizers,
and $E$ is a simple normal crossing divisor,
the locus of $[Y_0/A]$
where the canonical morphism $[Y_0/A] \to [Y_0/A]^\can$
is not an isomorphism
consists of disjoint union of
irreducible exceptional divisors.
Each of these irreducible divisors are smooth rational curves,
and Lemma \ref{lm:root1} is proved.
\end{proof}

\section{Cyclic case of Theorem \ref{th:global}
%Semiorthogonal decomposition for the canonical stack
}
 \label{sc:canonical}

Let $X$ be a surface with at worst quotient singularities and
consider the diagram
$$
\begin{CD}
 \scZ @>{q}>> \scX \\
 @V{p}VV @VV{\pi}V \\
 Y @>{\tau}>> X
\end{CD}
$$
where $\scX$ is the canonical stack associated with $X$,
$\tau : Y \to X$ is the minimal resolution, and
$\scZ$ is the reduced part of the fiber product $Y \times_{X} \scX$.
We consider the integral functor
$$
 \Phi:=q_* \circ p^*: D^b(\coh Y) \to D^b(\coh \scX),
$$
whose right adjoint will be denoted by $\Psi$.

\begin{proposition} \label{pr:orbifold}
Assume $X$ has only cyclic quotient singularities.
Then $\Phi$ is fully faithful and
there is a semiorthogonal decomposition
$$
 D^b \coh \scX =
  \langle
   E_1, \dots, E_{\ell},  \Phi (D^b \coh Y)
  \rangle
$$
where $E_1, \dots, E_{\ell}$ is an exceptional collection.
\end{proposition}

\begin{proof}
If $X$ is the global quotient
$\bA^2 / G$ for a finite small subgroup $G$ of $\GL_2(\bC)$,
the proof of \cite[Theorem 3.1]{Ishii_MKG} shows that
$\scZ$ is the quotient stack of the universal subscheme
in $Y \times \bA^2$ by the action of $G$
under the identification of $Y$ with $\GHilb(\bA^2)$.
Hence the assertion in this case follows
from Theorem \ref{th:cyclic}.

In the general case,
the composition $\Psi \circ \Phi$ is an integral functor
with respect to some kernel object $\scP$ on $Y \times Y$.
By the local case above,
$\scP$ is etale locally the structure sheaf of the diagonal.
Hence the kernel of $\Psi \circ \Phi$ is
a line bundle on the diagonal,
which implies that $\Phi$ is fully faithful.
Since the singularities of $X$ are isolated,
the semiorthogonal decomposition comes
from local contributions
around each singular points,
where the assertion holds
by the global quotient case above.
\end{proof}

This yields Step 4 in Introduction.

\section{Proofs of Theorem \ref{th:main} and Theorem \ref{th:global}}
 \label{sc:proof}

To prove Theorem \ref{th:main},
it remains to show the compatibility
between the functor
$$
 \Phi' : D^b \coh Y \to D^b \coh [\bA^2/G]
$$
and the composition
\begin{align*}
 D^b \coh Y
  \xto{\ \Phi_3\ } D^b \coh Y_2
  \xto{\ \Phi_2\ } D^b \coh \scY_1
  \xto{\ \Phi_1\ } D^b \coh [Y_0/A]
  \xto{\ \Phi_0\ } D^b \coh [\bA^2/G].
\end{align*}
It suffices to show that the functor
$\Phi_0 \circ \Phi_1 \circ \Phi_2$ is isomorphic
to the functor defined by the universal family
parameterized by the moduli space $Y_2=A\operatorname{-Hilb}(G_0\operatorname{-Hilb}(\bA^2))$.
Since $\Phi_2$ is the integral functor defined by the kernel object $\scO_{(Y_2 \times_{Y_1} \scY_1)_\text{red}}$ and $\Phi_1$ is the pull-back functor,
$\Phi_1 \circ \Phi_2$ is an integral functor whose kernel object is the pull-back of $\scO_{(Y_2 \times_{Y_1} \scY_1)_\text{red}}$ by the flat morphism $Y_2 \times [Y_0/A] \to Y_2 \times \scY_1$.
This object is isomorphic to $\scO_{(Y_2 \times [Y_0/A])_{\text{red}}}$ and therefore
$\Phi_1 \circ \Phi_2$ is isomorphic to the functor defined by the universal family of
$Y_2=\AHilb(Y_0)$.
Then by \cite[Lemma 2.2]{MR3049308},
$\Phi_0 \circ \Phi_1 \circ \Phi_2$ is isomorphic
to the functor defined by the universal family
parameterized by the moduli space $Y_2=\AHilb(\GoHilb(\bA^2))$.

To prove Theorem \ref{th:global},
we want to replace Theorem \ref{th:cyclic}
with Theorem \ref{th:main}
in the proof of Proposition \ref{pr:orbifold}.
In order to do that, consider the resolution $\tau_2: Y_2 \to X$
obtained by successively blowing up $Y$ so that
it is isomorphic to $\AHilb(\GoHilb(\bA^2))$
over an etale neighbourhood of a singular point of $X$
(whose corresponding point in $\scX$ has stabilizer group $G$).
Let $\widetilde{\scZ}$ be the reduced part of the fiber product
$Y_2 \times_X \scX$.
By the lemma below, for each singular point $P$ of $X$
whose neighbourhood is the quotient by a group $G_P \subset \GL(2, \bC)$,
there is a sheaf $\scF_P$ on $Y_2 \times \scX$
supported on $\tau^{-1}(P) \times BG_P$
with an extension
$$
0 \to \scO_{\widetilde{\scZ}} \to \scE \to \bigoplus_P \scF_P \to 0
$$
such that $\scE$ is isomorphic to the
universal family parameterized by $\AHilb(\GoHilb(\bA^2))$
for $G=G_P$ in an etale neighbourhood of each $P$.
If we define $\Phi$ as the integral functor whose kernel object is $\scE$,
then we can argue as in the proof of Proposition \ref{pr:orbifold}
to obtain Theorem \ref{th:global}.

\begin{lemma}
Consider the local situation $X=\bA^2/G$ with $Y=\GHilb(\bA^2)$
and $Y_2=\AHilb(\GoHilb(\bA^2))$.
Let $\widetilde{\scZ}$ be the reduced part of the figer product
$Y_2 \times_X \bA^2$ and let $\scE$ be the universal family
parameterized by $Y_2$.
Then there is an exact sequence
$$
0 \to \scO_{\widetilde{\scZ}} \to \scE \to \scF \to 0
$$
where $\scF$ is a $G$-equivariant coherent sheaf on $Y_2 \times \bA^2$
supported on $\tau_2^{-1}(0) \times \{0\}$.
\end{lemma}
\begin{proof}
The universal sheaf $\scE$ is the structure sheaf of the reduced part
of the fiber product $Y_2 \times_{Y_1} Y_0 \times_{\bA^2/G_0} \bA^2$
and there is a morphism
$Y_2 \times_{Y_1} Y_0 \times_{\bA^2/G_0} \bA^2 \to \scZ$
which implies a map $\scO_{\widetilde{\scZ}} \to \scE$.
It is an isomorphism over the smooth locus of $X$ and
since $\scO_{\widetilde{\scZ}}$ is torsion-free as a
coherent sheaf of $\scO_{Y_2}$-modules, this map is injective.
\end{proof}

\section{Invertible polynomials}
 \label{sc:invertible}

Let $n$ be a positive integer.
An integer $n \times n$-matrix
$A = (a_{ij})_{i, j = 1}^n$ with non-zero determinant
gives a polynomial $W \in \bC[x_1, \dots, x_n]$ by
$$
 W = \sum_{i = 1}^n x_1^{a_{i 1}} \cdots x_n^{a_{in}}.
$$
Non-zero coefficients of $W$
can be absorbed by rescaling $x_i$.
A polynomial obtained in this way
is called an {\em invertible polynomial}
if it has an isolated critical point at the origin.
Invertible polynomials play essential role
in transposition mirror symmetry
of Berglund and H\"{u}bsch
\cite{Berglund-Hubsch},
which has attracted much attention recently
(cf. e.g. \cite{Borisov_BHMSVA, Chiodo-Ruan_LGCYSSI, Krawitz,Takahashi_WPL}
and references therein).
The quotient ring
$R = \bC[x_1, \dots, x_n] / (W)$ is naturally graded
by the abelian group $L$
generated by $n + 1$ elements $\vecx_i$ and $\vecc$ with relations
$$
 a_{i 1} \vecx_1 + \dots + a_{i n} \vecx_n = \vecc,
  \qquad i = 1, \dots, n.
$$
The abelian group $L$ is the group of characters of $K$
defined by
$$
 K =
 \{ (\alpha_1, \dots, \alpha_n) \in (\bCx)^n
       \mid \alpha_1^{a_{11}} \cdots \alpha_n^{a_{1n}}
              = \dots = \alpha_1^{a_{n1}} \cdots \alpha_n^{a_{nn}}
 \}.
$$
The group $\Gmax$ of {\em maximal diagonal symmetries}
is defined as the kernel of the map
$$
\begin{array}{ccc}
 K & \to & \bCx \\
 \vin & & \vin \\
 (\alpha_1, \dots, \alpha_n) & \mapsto &
 \alpha_1^{a_{11}} \cdots \alpha_n^{a_{1n}},
\end{array}
$$
so that
there is an exact sequence
$$
 1 \to \Gmax \to K \to \bCx \to 1
$$
of abelian groups.
%This exact sequence induces an exact sequence
%$$
% 1 \to \bZ \to L \to \Gmax^\vee \to 1
%$$
%of the corresponding character groups,
%where
%$$
% \Gmax^\vee = \Hom(\Gmax, \bCx)
%$$
%is non-canonically isomorphic to $\Gmax$.
%If we write
%$$
% A^{-1} =
% \begin{pmatrix}
%    \varphi_1^{(1)} & \varphi_1^{(2)} & \cdots & \varphi_1^{(n)}\\
%    \varphi_2^{(1)} & \varphi_2^{(2)} & \cdots & \varphi_2^{(n)}\\
%    \vdots & \vdots & \ddots & \vdots \\
%    \varphi_n^{(1)} & \varphi_n^{(2)} & \cdots & \varphi_n^{(n)}\\
% \end{pmatrix},
%$$
%then the group $\Gmax$ is generated by
%$$
% \rho_k
%  = \lb \exp \lb 2 \pi \sqrt{-1} \varphi_1^{(k)} \rb, \dots, 
%      \exp \lb 2 \pi \sqrt{-1} \varphi_n^{(k)} \rb \rb,
%  \qquad k = 1, \dots, n.
%$$
%%Put
%%$$
%% \varphi_i = \varphi_i^{(1)} + \cdots + \varphi_i^{(n)},
%%  \qquad i = 1, \dots, n
%%$$
%%and define a homomorphism
%%$$
%% \varphi : \bCx \to K
%%$$
%%by
%%$$
%% \varphi(\alpha) = 
%%  \lb
%%   \alpha^{\ell \varphi_1}, \dots, \alpha^{\ell \varphi_n}
%%  \rb,
%%$$
%%where $\ell$ is the smallest integer
%%such that $\ell \varphi_i \in \bZ$ for $i = 1, \dots, n$.
%%Then $\varphi$ is injective and one has an exact sequence
%%$$
%% 1 \to \bCx \xto{\varphi} K \to \Gmaxbar \to 1,
%%$$
%%where $\Gmaxbar := \coker \varphi$.
%%$W$ is quasi-homogeneous of degree $\ell$
%%with respect to the $\bZ$-grading
%%$$
%% \deg x_i = \ell \varphi_i, \qquad i = 1, \dots, n.
%%$$
%%The intersection
%%$
%% \Image \varphi \cap \Gmax
%%$
%%is generated by
%%$$
%% J =
%% \lb \exp \lb 2 \pi \sqrt{-1} \varphi_1 \rb, \dots, 
%%      \exp \lb 2 \pi \sqrt{-1} \varphi_n \rb \rb.
%%$$
Let
$$
 \scX = [(W^{-1}(0) \setminus \{ 0 \}) / K]
$$
be the quotient stack of $W^{-1}(0) \setminus \{ 0 \}$
by the natural action of $K$.
It is a smooth Deligne-Mumford stack
since $W$ has an isolated critical point at the origin
and the action of $K$ at any point
in $W^{-1}(0) \setminus \{ 0 \}$
has a finite isotropy group.

\begin{lemma} \label{lem:rational}
The coarse moduli space of $\scX$ is a rational variety.
Moreover,
each codimension one irreducible component
of the locus where $\scX$ has non-trivial stabilizers
is also rational
and these components form a simple normal crossing divisor.
\end{lemma}

\begin{proof}
%(Old version)
%Consider the open dense subset
%$
% (W^{-1}(0) \cap (\bCx)^n)/K
%$
%of the coarse moduli space.
%Since
%$
% (W^{-1}(0) \cap (\bCx)^n)/\Gmax
%  \subset (\bCx)^n/\Gmax \cong (\bCx)^n
%$
%is a linear hyperplane,
%$
% (W^{-1}(0) \cap (\bCx)^n)/K
%$
%is an affine linear hyperplane in
%$
% (\bCx)^n/\bCx \cong (\bCx)^{n-1},
%$
%which is a rational variety.
%The locus where $\scX$ has non-trivial stabilizers is
%contained in the image of
%$
% \cup_i (W^{-1}(0) \cap \{x_i=0\}).
%$
%For a fixed $i$,
%$
% W^{-1}(0) \cap \{ x_i = 0 \}
%$
%is contained in the union of the closure of
%$
% W^{-1}(0) \cap \{x_i=0\}
%  \cap \{ x_j \ne 0 \text{ for } j \ne i\}
%$
%and $\{ x_i = x_l = 0 \}$ for $l \ne i$.
%The quotient of the former is
%also the closure of a hyperplane
%in a tours and the quotients of the latters are toric varieties.
%Hence they are rational.
%\end{proof}
%
%\begin{proof}
%(New version)
Since the $K$-action
on $(\bCx)^n$ is free,
the open dense substack
$$
 \scU = [(W^{-1}(0) \cap (\bCx)^n)/K]
$$
of $\scX$ is a scheme,
which is an affine linear subspace of
$$
 [(\bCx)^n / K] \cong (\bCx)^{n-1}
$$
considered as an open subscheme of $\bC^{n-1}$.
This shows that $\scX$ is rational.
A divisor with a non-trivial generic stabilizer
is the closure of either
$$
 W^{-1}(0) \cap \{ x_i = 0 \}
  \cap \{ x_k \ne 0 \text{ for } k \ne i\}
$$
for some $i$ or
$$
 \{ x_i = x_j = 0 \}
  \cap \{ x_k \ne 0 \text{ for } k \ne i, j \}
$$
for some $i \ne j$.
(If
$\{ x_i = x_j = 0 \}$ is not contained in $W^{-1}(0)$,
%$\{ x_i = x_j = 0 \} \not \subset W^{-1}(0)$,
then $W^{-1}(0) \cap \{ x_i = x_j = 0 \}$
has codimension greater than one.)
The quotient of the former also contains
an affine subspace of a tours, and
the quotient of the latter is a toric stack.
Hence they are rational.

Since the stabilizer group of any point on $\scX$ are abelian
and therefore locally diagonalizable,
the union of such divisors has normal crossings.
Moreover, at each point on the union,
different local components have
different stabilizer subgroups in $K$.
Hence the union has simple normal crossings.
\end{proof}

Now assume $n = 4$ so that $\dim \scX = 2$ and
let $Y$ be the minimal resolution
of the coarse moduli space of $\scX$.
%One can see from the description
%of a divisor with a generic stabilizer
%in the proof of Lemma \ref{lem:rational}
%that the union of divisors with generic stabilizers
%has simple normal crossings.
Since $Y$ is a rational surface,
one has a full exceptional collection on $Y$
by Orlov \cite{Orlov_PB}.
Let $\scX^{\can}$ be the canonical stack
associated with the coarse moduli space of $\scX$.
Then Theorem \ref{th:global} gives a full exceptional collection
on $\scX^{\can}$.
Since $\scX$ can be obtained
by successive root constructions
from $\scX^\can$,
Proposition \ref{pr:root} and Lemma \ref{lem:rational}
give the following:

\begin{theorem} \label{th:invertible}
The two-dimensional Deligne-Mumford stack
associated with an invertible polynomial in four variables
has a full exceptional collection.
\end{theorem}

\bibliographystyle{amsalpha}
\bibliography{bibs}

\noindent
Akira Ishii

Department of Mathematics,
Graduate School of Science,
Hiroshima University,
1-3-1 Kagamiyama,
Higashi-Hiroshima,
739-8526,
Japan

{\em e-mail address}\ : \ akira@math.sci.hiroshima-u.ac.jp

\ \\

\noindent
Kazushi Ueda

Department of Mathematics,
Graduate School of Science,
Osaka University,
Machikaneyama 1-1,
Toyonaka,
Osaka,
560-0043,
Japan.

{\em e-mail address}\ : \  kazushi@math.sci.osaka-u.ac.jp

\end{document}